\documentclass[a4paper,12pt]{amsart}
\usepackage{amssymb,latexsym}
\usepackage{enumerate}
\usepackage{amsmath,amsthm,mathrsfs}
\usepackage{graphicx}
\nonstopmode \makeatletter
\@namedef{subjclassname@2010}{%
\textup{2010} Mathematics Subject Classification} \makeatother

\numberwithin{equation}{section}
\newtheorem{thm}{Theorem}[section]
\newtheorem{cor}[thm]{Corollary}
\newtheorem{lem}[thm]{Lemma}
\newtheorem{mainthm}[thm]{Main Theorem}
\theoremstyle{definition}

\newcommand{\A}{{\mathscr A}}
\newcommand{\K}{{\mathscr K}}
\newcommand{\es}{{\mathscr S}}
\newcommand{\C}{{\mathbb C}}

\newcommand{\D}{{\mathbb D}}
\newcommand{\W}{{\mathscr W}}
\newcommand{\sphere}{{\widehat{\mathbb C}}}
\newcommand{\UCV}{{\mathcal{UCV}}}

\renewcommand{\Re}{{\operatorname{Re}\,}}

\newcommand{\inv}{^{-1}}
\newcommand{\T}{{\mathcal T}}

\begin{document}
\baselineskip=17pt

\title[Norm estimates of Schwarzian derivative]{
Sharp norm estimate of Schwarzian derivative for a class of convex functions
}
\author[S.~Kanas]{Stanis\l awa Kanas}
\address{Department of Mathematics\\ Rzesz\'ow University of
Technology\\
W. Pola 2, PL-35-959 Rzesz\'ow,  Poland}
\email{skanas@prz.rzeszow.pl}
\author[T. Sugawa]{Toshiyuki Sugawa}
\address{Graduate School of Information Sciences\\
Tohoku University\\ Aoba-ku, Sendai 980-8579, Japan}
\email{sugawa@math.is.tohoku.ac.jp}

\date{}

\renewcommand{\thefootnote}{}



\renewcommand{\thefootnote}{\arabic{footnote}}
\setcounter{footnote}{0}

\maketitle

\begin{abstract}
We establish a sharp norm estimate of the Schwarzian derivative for
a function in the classes of convex functions introduced by Ma and
Minda [Proceedings of the Conference on Complex Analysis,
International Press Inc., 1992, 157-169]. As applications, we give
sharp norm estimates for strongly convex functions of order
$\alpha,~0<\alpha<1,$ and for uniformly convex functions.
\end{abstract}

\maketitle

\section{Background and main result}

Let $\A$ be the class of analytic functions $f$ on the
unit disk $\D=\{z\in\C: |z|<1\}$ satisfying the normalization conditions
$f(0)=0$ and $f'(0)=1$
and let $\es$ be the class of univalent functions in $\A.$
The Schwarzian derivative
\begin{equation*}
S_f=\left(\frac{f''}{f'}\right)'-\frac12\left(\frac{f''}{f'}\right)^2
=\frac{f'''}{f'}-\frac32\left(\frac{f''}{f'}\right)^2
\end{equation*}
and its norm (the hyperbolic sup-norm)
$$
\|S_f\|=\sup_{z\in\D}(1-|z|^2)^2|S_f(z)|
$$
play an important role in the theory of Teichm\"uller spaces. Key
results concerning the Schwarzian derivative are summarized in the
following theorem.

\begin{thm}[Nehari \cite{Nehari49}, K\"uhnau \cite{Kuh71},
Ahlfors-Weill \cite{AW62}]\label{Thm:A}
Let $f\in\A.$
If $f$ is univalent, then $\|S_f\|\le 6.$
Conversely, if $\|S_f\|\le2,$ then $f$ is univalent.
Moreover, let $0\le k<1.$
If $f$ extends to a $k$-quasiconformal mapping of the Riemann sphere
$\sphere$ then $\|S_f\|\le 6k.$
Conversely, if $\|S_f\|\le 2k,$ then $f$ extends to
a $k$-quasiconformal mapping of $\sphere.$
\end{thm}

Here, a mapping $f:\sphere\to\sphere$ of the Riemann sphere
$\sphere=\C\cup\{\infty\}$ is called $k$-quasiconformal if
$f$ is a sense-preserving homeomorphism of $\sphere$ and
has locally integrable partial derivatives on $\C\setminus\{f\inv(\infty)\}$
with $|f_{\bar z}|\le k|f_z|$ a.e.
The best reference to the above theorem is Lehto's book \cite{Lehto:univ}.

The universal Teichm\"uller space $\T$ can be identified with the
set of Schwarzian derivatives of univalent analytic functions on
$\D$ with quasiconformal extensions to $\sphere.$ It is known that
$\T$ is a bounded domain in the Banach space of analytic functions
on $\D$ with finite hyperbolic sup norm (see \cite{Lehto:univ}).

In connection with Teichm\"uller spaces, it is an interesting problem
to consider estimation of the norm of the Schwarzian derivatives
of typical subclasses of univalent functions.
A function $f\in\A$ is called {\it starlike} (resp.~{\it convex})
if $f$ is univalent and the image $f(\D)$ is starlike with respect to
the origin (resp.~convex).
The classes of starlike and convex functions are denoted by
$\es^*$ and $\K,$ respectively.
It is well known that $f\in\A$ is starlike (resp.~convex) if and only if
$\Re[zf'(z)/f(z)]>0$ (resp.~$\Re[1+zf''(z)/f'(z)]>0$).
These notions have been refined and generalized in many ways
(see \cite{Goodman:univ}).
In the present note, we are mainly concerned with strongly starlike and convex
functions.
A function $f\in\A$ is called {\it strongly starlike}
(resp.~{\it strongly convex}) of order $\alpha~(0<\alpha<1)$
if $|\arg[zf'(z)/f(z)]|<\pi\alpha/2$
(resp.~$|\arg[1+zf''(z)/f'(z)]|<\pi\alpha/2$) in $|z|<1.$
The classes of strongly starlike and convex functions of order $\alpha$
will be denoted by $\es^*_\alpha$ and $\K_\alpha,$ respectively.
See \cite{SugawaDual} for geometric characterizations of functions
in $\es^*_\alpha.$

Define $\gamma(\beta)$ for $0<\beta<1$ by
$$
\gamma(\beta)=\frac2\pi\arctan\left[
\tan\frac{\pi\beta}{2}+\frac\beta{
(1+\beta)^{\frac{1+\beta}2}(1-\beta)^{\frac{1-\beta}2}\cos(\pi\beta/2)}\right].
$$
Note that $\gamma(\beta)$ increases from 0 to 1 when $\beta$
moves from 0 to 1.
Mocanu \cite{Mocanu89} found the following relation.

\begin{thm}[Mocanu]\label{Thm:Mocanu}
$\K_{\gamma(\beta)}\subset\es^*_{\beta}$ for $0<\beta<1.$
\end{thm}

In other words, $\K_\alpha\subset\es^*_{\gamma\inv(\alpha)}$ for $0<\alpha<1,$
where $\gamma\inv$ denotes the inverse function of $\gamma:[0,1]\to[0,1].$
For sharp or improved relations of this kind, see a paper \cite{KS07ss} of
the authors.

We summarize important properties of strongly starlike functions as follows.

\begin{thm}\label{Thm:B}
A strongly starlike function $f$ of order $\alpha\in(0,1)$ extends to
a $\sin(\pi\alpha/2)$-quasiconformal mapping of $\sphere$
and therefore $\|S_f\|\le 6\sin(\pi\alpha/2).$
\end{thm}

The first part is due to Fait, Krzy\.z and Zygmunt \cite{FKZ76} and
the second one is obtained from the first in combination with
Theorem \ref{Thm:A} (as was pointed out by Chiang \cite{Chiang91}).

By Theorems \ref{Thm:Mocanu} and \ref{Thm:B}, we see that a function
$f\in\K_\alpha$ extends to a
$\sin(\pi\gamma\inv(\alpha)/2)$-quasiconformal mapping of $\sphere$
and satisfies $\|S_f\|\le 6\sin(\pi\gamma\inv(\alpha)/2).$ On the
other hand, we have the following norm estimate for convex
functions.

\begin{thm}\label{Thm:C}
A convex function $f$ satisfies $\|S_f\|\le 2.$
The bound is sharp.
\end{thm}

This result was repeatedly proved in the literature
(see \cite{Rob69}, \cite{Nehari76}, \cite{Lehto77}),
and was refined by Suita \cite{Suita96} in the following form:
the {\it sharp} ienequality
$\|S_f\|\le 8\alpha(1-\alpha)$ holds for a function $f\in\A$
with $\Re[1+zf''(z)/f'(z)]>\alpha$ and $1/2\le\alpha<1.$

Obviously, the estimate $\|S_f\|\le 6\sin(\pi\gamma\inv(\alpha)/2)$
for $f\in\K_\alpha$ is not better than Theorem \ref{Thm:C} when
$\alpha$ is close to $1.$ We will give a sharp norm estimate for
$f\in\K_\alpha.$

\begin{mainthm}\label{thm:main}
Let $f$ be a strongly convex function of order $\alpha$ for
$0<\alpha<1.$ Then the sharp inequality $\|S_f\|\le 2\alpha$ holds.
\end{mainthm}

Define a function $f_\alpha\in\K_\alpha$ by the relation
$$
1+\frac{zf_\alpha''(z)}{f_\alpha'(z)}=\left(\frac{1+z^2}{1-z^2}\right)^\alpha.
$$
Then a simple computation gives
$$f_\alpha(z)=z+\alpha z^3/3+\alpha^2z^5/5+\alpha(1+8\alpha^2)z^7/63+
\cdots$$ and thus $S_{f_\alpha}(0)=2\alpha.$ Therefore, we see that
this satisfies $\|S_{f_\alpha}\|=2\alpha.$

Combining Theorem \ref{thm:main} with the Ahlfors-Weill theorem
(Theorem \ref{Thm:A}), we obtain the following result.

\begin{cor}
A function $f\in\K_\alpha$ extends to
an $\alpha$-quasiconformal mapping of $\sphere$ for $0<\alpha<1.$
\end{cor}

By using Mathematica Ver.~7, we found that $\sin(\pi\gamma\inv(\alpha)/2)
<\alpha$ when $0<\alpha<0.3354$ (see Figure 1).
Therefore, the corollary gives a better bound only when $\alpha>0.3355,$
though it has obvious merit of simplicity.

\begin{figure}[htbp]
\begin{center}
\includegraphics[height=0.25\textheight, bb=0 0 240 143]{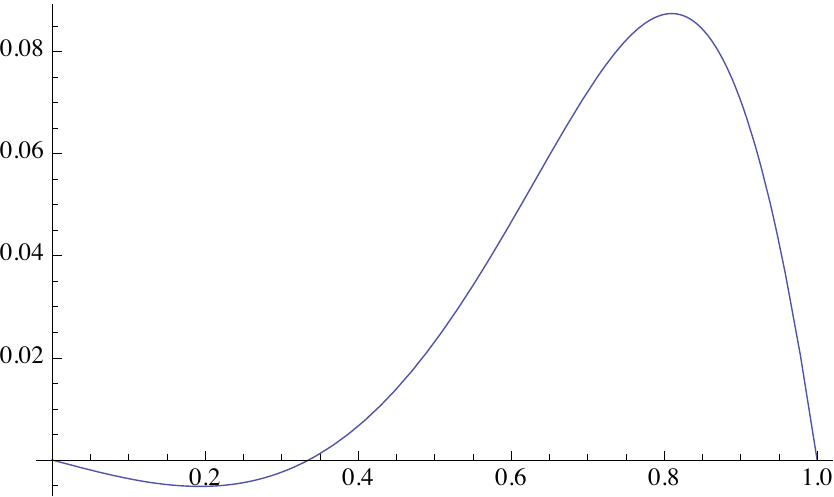}
\end{center}
\caption{Graph of $\sin(\pi\gamma\inv(\alpha)/2)-\alpha$}
\end{figure}

For some reason, the second author was even led in \cite{SugawaST}
to the expectation that each function in $\es^*_\alpha$ might extend
to an $\alpha$-quasiconformal mapping of $\sphere.$ This was
recently disproved by Yuliang Shen \cite{Shen07} for every
$0<\alpha<1.$

Goodman \cite{Good91} introduced the class $\UCV$ of uniformly
convex functions. Here, a function $f\in\A$ is called {\it uniformly
convex} if every (positively oriented) circular arc of the form
$\{z\in\D: |z-\zeta|=r\},~\zeta\in\D,~0<r<|\zeta|+1,$ is mapped by
$f$ univalently onto a convex arc.
In particular, $\UCV\subset\K.$
See also \cite{KW99}, \cite{KS05} and \cite{Kanas09}
for $k$-uniform convexity ($0\le k<+\infty$),
a more refined notion of convexity, and related results.
We have the following sharp norm estimate for $\UCV.$

\begin{mainthm}\label{thm:main2}
Let $f$ be a uniformly convex function. Then the sharp inequality
$\|S_f\|\le \frac{8}{\pi^2}$ holds. In particular, $f$ extends to a
$4/\pi^2$-quasiconformal mapping of $\sphere.$
\end{mainthm}

In \cite{Kanas09} the first author observed that a uniformly convex function
extends to a $k_1$-quasiconformal mapping of $\sphere,$
where $k_1=\sin(\pi\gamma\inv(1/2)/2)\approx 0.52311.$
Therefore,  the above bound $4/\pi^2\approx 0.40528$ is slightly better.
(Note also that a numerical computation of $K(1)=(1+k_1)/(1-k_1)
\approx 3.19387$ was not correct in \cite{Kanas09}.)

In Section 2, we provide a principle leading to a sharp norm estimate
of the Schwarzian derivative for a subclass of $\K$ given in a specific way.
By making use of it, we prove Theorem \ref{thm:main} in Section 3
and Theorem \ref{thm:main2} in Section 4.

\section{General norm estimate for convex functions}

Ma and Minda \cite{MM92B} introduced a unifying way of treatment
of various subclasses of $\K.$
Let $\varphi$ be an analytic function on $\D$ with $\varphi(0)=1.$
The class $\K(\varphi)$ is defined to be the set of functions $f\in\A$
with $1+zf''(z)/f'(z)\prec \varphi(z).$
Here, an analytic function $g$ on $\D$ is said to be subordinate
to another $h$ and denoted by $g\prec h$ or $g(z)\prec h(z)$
if $g=h\circ\omega$ for an analytic function $\omega$ on $\D$
with $\omega(0)=0$ and $|\omega|<1.$
When $h$ is univalent, it is useful to note the following fact:
$g\prec h$ if and only if $g(0)=h(0)$ and $g(\D)\subset h(\D).$

Let
$$
P_\alpha(z)=\left(\frac{1+z}{1-z}\right)^\alpha
$$
for a constant $\alpha>0.$
Then $P_\alpha$ maps $\D$ univalently onto the sector $|\arg w|<\pi\alpha/2$
for $0<\alpha\le1.$
Therefore, $\K(P_\alpha)=\K_\alpha$ for $0<\alpha<1$ and $\K(P_1)=\K.$

Ma and Minda \cite{MM92U} and R\o nning \cite{Ron93} found the following
characterization of the class $\UCV.$
A function $f\in\A$ is uniformly convex if and only if
$\Re[1+zf''(z)/f'(z)]>|zf''(z)/f'(z)|,~z\in\D.$
Noting the fact that the function
\begin{equation}\label{eq:P}
P(z)=1+\frac{2}{\pi^2}\left(\log\frac{1+\sqrt z}{1-\sqrt z}\right)^2
\end{equation}
maps $\D$ univalently onto the domain $\{w: \Re w>|w-1|\},$
we have $\UCV=\K(P)$ (see \cite[p.~191]{Ron93}).

We first formulate the sharp norm estimate for the class $\K(\varphi).$
To this end, we consider the quantity
$$
F(s,t)=\frac{(1-t^2)^2}{2t^2}A(s)
+(1-t^2)\left(1-\frac{s^2}{t^2}\right)B(s),
$$
where
$$
A(s)=\sup_{|z|=s}\big|2z\varphi'(z)+1-\varphi(z)^2\big|,
\quad\text{and}\quad
B(s)=\sup_{|z|=s}\big|\varphi'(z)\big|.
$$
Define
$$
N(\varphi)=\sup_{0<s<t<1}F(s,t).
$$
Then we have the following.

\begin{mainthm}\label{thm:general}
Let $\varphi$ be an analytic function on the unit disk with $\varphi(0)=1.$
Then, the sharp inequality
$\|S_f\|\le N(\varphi)$ holds for $f\in\K(\varphi).$
\end{mainthm}

Note here that $\varphi$ is not required to satisfy $\Re\varphi>0$
in the theorem, though there is no guarantee for finiteness of $N(\varphi)$
in this general case.

\begin{proof}
Denote by $\W$ the set of analytic functions $\omega$ on $\D$ with
$\omega(0)=0$ and $|\omega|<1.$

Let $f\in\K(\varphi).$
Since $1+zf''(z)/f'(z)\prec \varphi(z)$ by definition,
we have $f''(z)/f'(z)=(\varphi(\omega(z))-1)/z$
for an $\omega\in\W.$
Set $w=\omega(z)$ for a fixed $z\in\D.$
Then the Schwarzian derivative $S_f$ of $f$ can be expressed by
$$
S_f(z)=
\frac{\varphi'(w)\omega'(z)}z
-\frac{\varphi(w)^2-1}{2z^2}.
$$
We now recall Dieudonn\'e's lemma (cf.~\cite{Duren:univ}): for a fixed pair
of points $z, w\in\D$ with $|w|\le |z|,$ one has
$$
\left\{\omega'(z): \omega\in\W, \omega(z)=w \right\}
=\left\{v\in\C:
\left|v-\frac{w}{z}\right|\le\frac{|z|^2-|w|^2}{|z|(1-|z|^2)}\right\}.
$$
This means that $\zeta=\omega'(z)-w/z$ varies over the closed disk
$$|\zeta|\le (t^2-|w|^2)/t(1-t^2)$$ for fixed $|z|=t<1.$ Then we can
write
\begin{align*}
S_f(z)&=
\frac{\varphi'(w)}z\left(\zeta+\frac wz\right)
-\frac{\varphi(w)^2-1}{2z^2} \\
&=\frac{2w\varphi'(w)+1-\varphi(w)^2}{2z^2}+\zeta\cdot\frac{\varphi'(w)}{z}.
\end{align*}
Therefore, for $|z|=t<1,$ we have the sharp inequality
$$
|S_f(z)|\le
\frac{|2w\varphi'(w)+1-\varphi(w)^2|}{2t^2}+
\frac{t^2-|w|^2}{t^2(1-t^2)}|\varphi'(w)|
$$
for $f\in\K(\varphi).$
This fact can be expressed by
\begin{align*}
\sup_{f\in\K(\varphi)}|S_f(z)|
&=\sup_{s\le t}\sup_{|w|=s}
\left[\frac{|2w\varphi'(w)+1-\varphi(w)^2|}{2t^2}+
\frac{t^2-|w|^2}{t^2(1-t^2)}|\varphi'(w)|\right] \\
&=\sup_{s\le t}
\left[\frac{A(s)}{2t^2}+
\frac{t^2-s^2}{t^2(1-t^2)}B(s)\right].
\end{align*}
Hence, we have
$$
\sup_{f\in\K(\varphi)}(1-t^2)|S_f(z)|=\sup_{s<t}F(s,t)
$$
for a fixed point $z\in\D$ with $|z|=t$ and the required relation.
\end{proof}

As we will see below, we often have the relations
$$
A(s)=2s\varphi'(s)+1-\varphi(s)^2
\quad\text{and}\quad
B(s)=\varphi'(s)
$$
for $0\le s<1.$
Then, by a simple computation, we have the expression
\begin{equation}\label{eq:F}
F(s,t)=\frac{(1-t^2)^2}{2t^2}(1-\varphi(s)^2)
+\frac{(1-t^2)(1-s)(s+t^2)}{t^2}\varphi'(s).
\end{equation}

Observe that $F(s,t)$ is described in terms of $s$ and $t^2$ in this case.

\section{Proof of Theorem \ref{thm:main}}

We begin with the following properties of $P_\alpha.$

\begin{lem}\label{lem:P}
The functions $P_\alpha(z)$ and $Q_\alpha(z)=2zP_\alpha'(z)+1-P_\alpha(z)^2$
have non-negative Taylor coefficients about $z=0$
for $0<\alpha\le1.$
\end{lem}

\begin{proof}
Since
$$
\log\frac{1+z}{1-z}=2\sum_{n=1}^\infty\frac{z^{2n-1}}{2n-1},
$$
the function $P_\alpha(z)=\exp(\alpha\log\frac{1+z}{1-z})$ has
positive Taylor coefficients.
Note also, by this expression, that $P_\alpha$ satisfies the differential
equation
\begin{equation}\label{eq:de}
\frac{P_\alpha'(z)}{P_\alpha(z)}=\frac{2\alpha}{1-z^2},
\end{equation}
and therefore,
\begin{equation}\label{eq:de2}
P_\alpha''(z)=\frac{2(\alpha+z)}{1-z^2}P_\alpha'(z),
\end{equation}

We next observe the expansion
$$
P_\alpha(z)=1+\sum_{n=1}^\infty a_nz^n.
$$
Since $a_1=2\alpha$ and $P_\alpha$ maps $\D$
univalently onto a convex domain for $0<\alpha\le 1,$
by a theorem of L\"owner (cf.~\cite{Duren:univ}), we have
\begin{equation}\label{eq:an}
0\le a_n\le 2\alpha \quad (n=1,2,3,\dots).
\end{equation}
By using \eqref{eq:de} and \eqref{eq:de2}, we now have the expression
$$
Q_\alpha'(z)=2zP_\alpha''(z)+2P_\alpha'(z)(1-P_\alpha(z))
=2P_\alpha'(z)\left(1-P_\alpha(z)+\frac{2z(\alpha+z)}{1-z^2}\right).
$$
Since $P_\alpha'(z)$ has positive Taylor coefficients and
$$
1-P_\alpha(z)+\frac{2z(\alpha+z)}{1-z^2}
=\sum_{n=1}^\infty(2\alpha-a_{2n-1})z^{2n-1}
+\sum_{n=1}^\infty(2-a_{2n})z^{2n},
$$
the required assertion for $Q_\alpha$ is deduced from \eqref{eq:an}.
\end{proof}

We are now ready to prove our main theorem.

\begin{proof}[Proof of Theorem \ref{thm:main}]
In view of Theorem \ref{thm:general}, we need to show that
$N(P_\alpha)=2\alpha.$

By Lemma \ref{lem:P}, we can apply \eqref{eq:F} for $\varphi=P_\alpha:$
\begin{align*}
F(s,t)
&=\frac{(1-t^2)^2}{2t^2}(1-P_\alpha(s)^2)
+\frac{(1-t^2)(1-s)(s+t^2)}{t^2}P_\alpha'(s) \\
&=\frac{(1-t^2)^2}{2t^2}(1-P_\alpha(s)^2)
+2\alpha\frac{(1-t^2)(s+t^2)}{t^2(1+s)}P_\alpha(s).
\end{align*}
Here, we have used \eqref{eq:de}.

Since $F(0,t)=2\alpha(1-t^2)\to 2\alpha$ as $t\to0,$
it is enough to show the inequality $F(s,t)\le 2\alpha$ when
$0<s<t<1.$
Letting $x=1-t^2,$ we see now that
\begin{align*}
&F(s,t)\le 2\alpha \\
\Leftrightarrow \quad &
\left(1-P_\alpha(s)^2\right)x^2+\frac{4\alpha x(1+s-x)}{1+s}P_\alpha(s)
\le 4\alpha(1-x) \\
\Leftrightarrow \quad &
\left(P_\alpha(s)^2+\frac{4\alpha}{1+s}P_\alpha(s)-1\right)x^2
-4\alpha\big(1+P_\alpha(s)\big)x+4\alpha\ge0.
\end{align*}
The left-hand side in the last inequality
can be regarded as a quadratic polynomial in $x$ of the form
$$
Kx^2-4Mx+4L=K\left(x-\frac{2M}K\right)^2+\frac4K(KL-M^2)
$$
with $K>0.$
We now compute $KL-M^2$ as follows:
\begin{align*}
h(s):=&~\alpha\left(P_\alpha(s)^2+\frac{4\alpha}{1+s}P_\alpha(s)-1\right)
-\alpha^2\big(1+P_\alpha(s)\big)^2 \\
= &~\alpha\left((1-\alpha)P_\alpha(s)^2
+\frac{2\alpha(1-s)}{1+s}P_\alpha(s)-(1+\alpha)\right).
\end{align*}
Since
$$
h'(s)=\frac{4\alpha^2(1-\alpha)}{(1+s)^2}
\left(P_{\alpha+1}(s)-1\right)P_\alpha(s)>0
$$
for $s>0,$ the function $h(s)$ is monotone increasing in $0<s<1.$
Thus $h(s)>h(0)=0$ for $0<s<1.$
Therefore, we have seen that $KL-M^2\ge0,$ which implies $F(s,t)\le 2\alpha$
as expected.
\end{proof}

Obviously, the above proof covers the case when $\alpha=1.$ Thus, we
have obtained yet another proof of Theorem \ref{Thm:C}.

\section{Proof of Theorem \ref{thm:main2}}

We will need the following estimate.

\begin{lem}\label{lem:sum}
For every non-negative integer $n,$ the following inequality holds:
$$
\sum_{\substack{0\le k,l,m \\ k+l+m=n}}\frac1{(2k+1)(2l+1)(2m+1)}\le 1.
$$
\end{lem}

\begin{proof}
We denote by $A_n$ the summation in question.
Also, let
$$
B_n=\sum_{\substack{0\le k,l \\ k+l=n}}\frac1{(2k+1)(2l+1)}.
$$
By the partial fraction decomposition, we observe
\begin{align*}
B_n&=\frac1{2n+2}\sum_{k+l=n}\left(\frac1{2k+1}+\frac1{2l+1}\right) \\
&=\frac1{n+1}\sum_{k=0}^n\frac1{2k+1}.
\end{align*}
Since $1/(2k+1)\le 1/3$ for $k\ge1,$ we have
\begin{equation}\label{eq:Bn}
B_n\le \frac{1}{n+1}\left(1+\frac n3\right)\le \frac23
\quad \text{for}~n\ge1.
\end{equation}
Similarly, by the partial fraction decomposition, we have
\begin{align*}
&\frac{2n+3}{(2k+1)(2l+1)(2m+1)} \\
=&\frac1{(2k+1)(2l+1)}+\frac1{(2l+1)(2m+1)}
+\frac1{(2m+1)(2k+1)}
\end{align*}
for $k+l+m=n.$
We now apply it and take into account
the symmetry in $k, l, m$ and \eqref{eq:Bn} to obtain finally
$$
A_n
=\frac3{2n+3}\sum_{j=0}^n B_j
\le\frac3{2n+3}\left(1+\frac23 n\right)=1.
$$
\end{proof}

\noindent
{\bf Note.}
In the previous manuscript, we had only a lengthy proof for 
Lemma \ref{lem:sum}. The second author asked for an
elegant proof of it in {\it Sugaku Seminar},
a mathematical monthly magazine published in Japan.
Several readers gave nice proofs as above.
See an article (in Japanese) of the second author in
{\it ibid.}~vol.~{\bf 50}, no.~3 (2011) for details.
The authors would like to express their thanks to the
readers of the magazine.

\bigskip

We next show a result similar to Lemma \ref{lem:P}.

\begin{lem}\label{lem:P1}
The functions $P$ given in \eqref{eq:P} and
$Q(z)=2zP'(z)+1-P(z)^2$ have non-negative Taylor
coefficients about $z=0.$
\end{lem}

\begin{proof}
Let
$$
G(z)=\sum_{n=0}^\infty\frac{z^n}{2n+1}
=\frac1{2\sqrt z}\log\frac{1+\sqrt{z}}{1-\sqrt{z}}.
$$
Then
\begin{equation}\label{eq:P2}
P(z)=1+\frac{8}{\pi^2}zG(z)^2.
\end{equation}
Therefore, it is immediate to see that $P(z)$ has positive Taylor
coefficients about $z=0.$
Furtheremore, we can easily check the formula
\begin{equation}\label{eq:P'}
P'(z)=\frac{8G(z)}{\pi^2(1-z)}.
\end{equation}
We also note that $G^3$ can be described by
$$
G(z)^3=\sum_{n=0}^\infty A_nz^n,
$$
where $A_n$ is the number given in Lemma \ref{lem:sum}.
With these facts in mind, we now compute
\begin{align*}
Q(z)
&=\frac{16}{\pi^2}zG(z)\left(
\frac1{1-z}-G(z)-\frac{4z}{\pi^2}G(z)^3\right) \\
&=\frac{64}{\pi^4}z^2G(z)
\sum_{n=0}^\infty\left(\frac{\pi^2}{4}\cdot\frac{2n+2}{2n+3}-A_n\right)z^n.
\end{align*}
Since
$$
\frac{\pi^2}{4}\cdot\frac{2n+2}{2n+3}
\ge\frac{\pi^2}{6}>1\ge A_n
$$
for $n=0,1,2,\dots$ by Lemma \ref{lem:sum}, we see that
the Taylor coefficients of $Q$ about $z=0$ are non-negative.
\end{proof}

We are now in a position to show the second main result.

\begin{proof}[Proof of Theorem \ref{thm:main2}]
We take the same strategy as in the proof of Theorem \ref{thm:main}.
In view of Theorem \ref{thm:general}, we only need to show that
$N(P)=8/\pi^2.$ By Lemma \ref{lem:P1} and formulae \eqref{eq:F},
\eqref{eq:P2}, \eqref{eq:P'}, we now have
\begin{align*}
F(s,t)&=\frac{(1-t^2)^2}{2t^2}(1-P(s)^2)
+\frac{(1-t^2)(1-s)(s+t^2)}{t^2}P'(s) \\
&=\frac{8(1-t^2)}{\pi^2t^2}\left(
(s+t^2)G(s)-s(1-t^2)G(s)^2-\frac4{\pi^2}s^2(1-t^2)G(s)^3 \right) \\
&=\frac{8x}{\pi^2(1-x)}\left(
(s+1-x)G(s)-sxG(s)^2-\frac4{\pi^2}s^2xG(s)^3 \right),
\end{align*}
where we put $x=1-t^2.$
Since $F(0,t)=8(1-t^2)/\pi^2\to 8/\pi^2$ as $t\to0,$
it suffices to show that $F(s,t)\le 8/\pi^2$ for $0<s<t<1.$
This is equivalent to the inequality
\begin{align*}
&x\left((s+1-x)G(s)-sxG(s)^2-\frac4{\pi^2}s^2xG(s)^3 \right)
\le 1-x \\
\Leftrightarrow&
\left(G(s)+sG(s)^2+\frac4{\pi^2}G(s)^4\right)x^2
-\big(1+(1+s)G(s)\big)x+1\ge0
\end{align*}
for $0<x<1-s^2.$
The left-hand side in the last inequality is of the form
$Kx^2-Mx+L$ with
$$
4KL-M^2=\left(\frac{4sG(s)^2}{\pi}\right)^2
-\big(1-(1-s)G(s)\big)^2.
$$
It is enough to show $4KL-M^2\ge0.$
Since $G(s)<1/(1-s),$
we observe that $4KL-M^2\ge0$ if and only if
$$
\frac{4sG(s)^2}{\pi}\ge 1-(1-s)G(s),
$$
which is equivalent to
\begin{equation}\label{eq:last}
G(s)\ge \pi\frac{\sqrt{(1-s)^2+16s/\pi}-1+s}{8s}.
\end{equation}
Since it is easily checked that
$\pi(\sqrt{(1-s)^2+16s/\pi}-1+s)/8s<1$ and $G(s)>1$ for $0<s<1,$
the inequality \eqref{eq:last} certainly holds.
Thus the proof is now complete.
\end{proof}

Define a function $f_0\in\UCV$ by the relation
$$
1+\frac{zf_0''(z)}{f_0'(z)}=P(z^2)
=1+\frac{2}{\pi^2}\left(\log\frac{1+z}{1-z}\right)^2.
$$
Then we have
$$
f_0(z)=z+\frac{4}{3\pi^2}z^3+\left(\frac{4}{15\pi^2}+\frac{8}{5\pi^4}\right)z^5
+\cdots
$$
and thus $S_{f_0}(0)=\frac{8}{\pi^2}$ so that
$\|S_{f_0}\|=\frac{8}{\pi^2}.$

\subsection*{Acknowledgements}
The second author was supported in part by JSPS Core-to-Core Program
18005 and JSPS Grant-in-Aid for Exploratory Research, 19654027.

\end{document}